\documentclass[11pt,a4paper]{amsart}  

\numberwithin{equation}{section}
\allowdisplaybreaks

\usepackage{amssymb}
\usepackage[latin1]{inputenc}
\usepackage{graphicx}
\usepackage{amsthm}
\usepackage{amsmath}
\usepackage{color}

\newcommand{\Nset}{\mathbb{N}}

\newcommand{\Rset}{\mathbb{R}}
\newcommand{\Cset}{\mathbb{C}}
\newcommand\Ln{\mathop{\mathrm{ln}}}
\def\Arg{\mathop{\rm Arg}}

\def\ds{\displaystyle}

\begin{document}


\title[Volterra functions and Ramanujan integrals]{On Volterra functions \\ and  Ramanujan integrals}

\author[R. Garrappa, F. Mainardi]{Roberto Garrappa$^{1}$, Francesco Mainardi$^{2}$}


\address{$^{1}$ Department of Mathematics, University of Bari \\ Via E. Orabona 4, I-70125 Bari, Italy}
\email{roberto.garrappa@uniba.it}

\address{$^{2}$ Department of Physics and Astronomy, University of Bologna, and INFN \\ Via Irnerio 46, I-40126 Bologna, Italy}
\email{francesco.mainardi@bo.infn.it}

\date{September 2016 
\\ 
{\it Note}: This E-Print is a revised version with a different layout
 of the paper published  in\\
{\bf  Analysis,
Vol. 36 No~2  (2016), pp. 89--105.}  DOI: 10.1515/anly-2015-5009
\\
{\it Dedicated to the memory of Professor Anatoly Kilbas (1948--2010)}
}

 \newtheorem{theorem}{Theorem}[section]
 \newtheorem{lemma}[theorem]{Lemma}

\newtheorem{remark}[theorem]{Remark}

\begin{abstract}
Volterra functions were introduced at the beginning of the twentieth century as solutions of some integral equations of convolution type with logarithmic kernel. Since then, few authors have studied this family of functions and faced with the problem of providing a clear understanding of their asymptotic behaviour for small and large arguments. This paper reviews some of the most important results on Volterra functions and in particular collects, into a quite general framework, several results on their asymptotic expansions; these results turn out to be useful not only for the full understanding of the behaviour of the Volterra functions but also for their numerical computation. The connections with integrals of Ramanujan type, which have several important applications, are also discussed.
\end{abstract}

\maketitle

\section{Introduction}\label{S:Introduction}

In the Chapter 18, entitled ``Miscellaneous Functions'', of the third volume of the Bateman Handbook devoted to High Transcendental Functions  
\cite{ErdelyiMagnusOberhettingerTricomi3}  we find, in addition to the Mittag-Leffler and Wright functions (nowadays well known for their applications in Fractional Calculus and classified as Fox-H-functions), the so called Volterra functions. They are so named after  the great Italian mathematician Vito Volterra (1860-1940) who introduced  them, as solutions of some integral equations of convolution type with logarithmic kernels, in a 1916 article \cite{Volterra1916} and then in his 1924 book  together with Joseph P{\`e}r{\'e}s \cite{Volterra-Peres1924}. However, these functions appeared at the beginning of the second decade of the twentieth century also in the works of other famous mathematicians, including Srinivasa Ramanujan \cite{Hardy1940}, Jaques  Touchard  \cite{Touchard1913}, Edmund Landau \cite{Landau1918} and    in a variety of fields of mathematics, ranging from the theory of prime numbers to definite integrals and integral equations.
A related formula, nowadays referred to as ``the Ramanujan identity'', appeared presumably for the first time in a notebook  of 1913 of this outstanding Indian mathematician, as reported in the 1940 book by Hardy \cite{Hardy1940}.

In their most general formulation, and by following the notation of Colombo \cite {Colombo1955} adopted also in the Bateman handbook, the Volterra functions for real arguments $t>0$ read  
\begin{equation}\label{VolterraIntegralDefinition}
	\mu(t,\beta,\alpha) := \frac{1}{\Gamma(\beta+1)} \int_{0}^{\infty} \frac{t^{u+\alpha} u^{\beta}}{\Gamma(u+\alpha+1)} \, du,
	 \quad \Re(\beta) > -1\,,
	\end{equation}
but some particular notations are usually adopted in the special cases 
\begin{equation}\label{eq:VF_SpecialCases}
	\begin{array}{ll}
		\alpha=\beta=0 \,: \,& \nu(t) = \mu(t,0,0) \,,\\
		\alpha\not=0, \, \beta=0 \,: \, & \nu(t,\alpha)= \mu(t,0,\alpha)  \,, \\
		\alpha=0, \, \beta\not=0 \,: \, & \mu(t,\beta) = \mu(t,\beta,0)  \,.\\
	\end{array}
\end{equation}

Here, we recall the simplest integral equation formerly considered by Volterra which leads to the function $\nu(t)$. Using the modern notation, it is a convolution Volterra integral equation of the first kind with logarithmic kernel which reads  
\begin{equation}\label{eq:(1.3)}
	\int_0^t u(\tau) \, \log (t-\tau)\, d\tau = f(t)\,,
	\end{equation}
where $f(t)$  is an assigned, sufficiently well-behaved function 
with $f(0)=0$. 
According to  the original Volterra notation, this is a particular  case of the Volterra composition of the first kind referred to as integral equation of the closed cycle (see \cite{Volterra1916,Volterra-Peres1924}).
In our notation, the Volterra solution of (\ref{eq:(1.3)}) is   
\begin{equation}\label{eq:(1.4)}
  u(t) = - e^{-\gamma} \, \int_0^t \!\!
  \dot f (t-\tau) \, \dot \nu\left(\tau e^{-\gamma}\right)\, d\tau\,,
  \end{equation}
  where the superposed  dot to a function  denotes differentiation with respect to the argument and $\gamma =  0.57721...$ is the Euler-Mascheroni constant.
Nowadays the Volterra integral equations of convolution type (which include the Abel integral equations) are usually solved by means of the Laplace transformation.
  
  
  We incidentally find the function $\nu(t)$ in the table of inverse Laplace transforms in the second volume of the Batemat Handbook   devoted to Integral Transforms, which is also recalled in the recent papers by Mainardi et al.
\cite{MainardiMuraGorenfloSojanovic2007,MainardiMuraPagniniGorenflo2008}
on the solution of  fractional relaxation/diffusion equations of distributed order.
  
To the best knowledge of the authors, the Volterra  functions do not appear in any other Handbook on Special Functions and  Integral transforms published in English even if investigated to some extent in France between 1943-1953. In fact,  the Volterra functions have aroused a special interest of several French mathematicians (Pierre Humbert, Louis Poli, Maurice Parodi, Serge Colombo and Pierre Barrucand) who investigated them in the context of the Laplace transformation where the functions play an important role as direct and inverse transforms.

In the former Soviet Union the Volterra integral equations with logarithmic kernel have attracted the attention of 
Dzrbashyan  \cite{Dzrbashyan1959,Dzrbashyan1960} in the sixties,
and of Krasnov et al. \cite{Krasnov-et-al1977} and  
Kilbas \cite{Kilbas1976,Kilbas1977} in the seventies.   Other relevant contributions on these equations include the 1959 paper by Butzer \cite{Butzer1959} and the 1992 book by Srivastava and Buschman 
\cite{Srivastava-Buschman1992}.

Recently Alexander Apelblat, after his preliminary interaction with Serge Colombo, and a number of his own papers since 1985, 
 has published  in 2008 \cite{Apelblat2008} and in 2010 \cite{Apelblat2010}
  two books entirely devoted to Volterra functions providing  historical perspectives, a large bibliography  and  lists of identities and integral transforms related to these functions.
  The Apelblat book  of 2010 is  essentially  devoted to integral transforms related to Volterra functions, so that it can be considered an extraction from the previous larger treatise of 2008. 

In the present survey we cannot recall all the topics dealt in the excellent books by Apelblat, that up to nowadays remain the only relevant sources of these almost forgotten functions. However our main contribute is devoted to complement these treatises by presenting, in a quite general framework, the main results on the asymptotic behaviour of the Volterra functions and illustrate them by means of some plots obtained after numerical computation. We also study and present, in a more formal way, the generalization of the Ramanujan identities and functions which are strictly related to the Volterra functions and have some important and interesting applications in physics; this result is achieved after providing an explicit representation of the residues of Laplace transform of the Volterra functions (a results which is exploited also for the numerical computation). Although some of these results were already presented in the book of Apelblat for few instances of the second parameter $\beta$, here we provide a general formulation which holds for any admissible value of $\beta$.


 We hope that our analysis will be useful  for researchers  interested on possible applications of non local operators involving not only power law functions (as it occurs in pseudo differential operators of fractional calculus), but also slow varying functions  like logarithmic functions.

The plan of the paper is the following. In the next Section we recall the Laplace transform of the Volterra functions and we provide an explicit formula for the evaluation of its residue in order to highlight, in Section \ref{S:Ramanujan}, the connections of the Volterra functions $\mu(t, k,\alpha)$, for $k \in \Nset$ and $-1<\alpha\le0$, with the Ramanujian integral and some generalizations. We hence discuss in Section \ref{S:Asymptotic}  some asymptotic representations and we provide some comparisons with the results obtained by the numerical inversion of the Laplace transform. Finally, Section \ref{S:ConcludingRemarks} is devoted to some concluding remarks. As an instructive exercise of Laplace transforms, in the Appendix we show how to obtain the Volterra solution (\ref{eq:(1.4)}), so reducing the original Volterra approach to a student problem.

\section{Laplace transform of the Volterra functions and Ramanujan integrals}\label{S:LaplaceTransform}

The Laplace transform ${\mathcal M}(s,\beta,\alpha)$ of the Volterra function 
$\mu(t,\beta,\alpha)$ is 
\begin{equation}
	{\mathcal M}(s,\beta,\alpha) 
:= \int_0^\infty \!\! e^{-st}\, 	\mu(t,\beta,\alpha)\, dt  
	= \frac{1}{s^{\alpha+1} (\ln(s))^{\beta+1}},
	\end{equation}
	for  $\Re(\alpha)>-1, \, \Re(\beta) >-1, \, \Re(s) > 1$,
where  the introduction of a branch--cut on the real negative axis is necessary to make single--valued both the real power and the logarithmic function.
 However, when $\beta\not\in\Nset$ it is necessary to extend the cut on $(0,1]$ to avoid the multi--valued character of the composite function $(\ln(s))^{\beta+1}$. 

Therefore, for $\beta \in \Nset$ we must consider a branch--cut on $(-\infty,0]$ with a branch--point singularity at $s^{\star}=0$ and an isolated singularity at $s^{\star}=1$; otherwise, i.e. when $\beta \in \Rset-\Nset$, the branch-cut must be chosen on $(-\infty,1]$, with the branch--point singularity at $s^{\star}=1$. 

With the above selection of the branch-cut the single--valued counterpart of the complex logarithmic function is hence
\begin{equation}\label{eq:ComplexLog}
	\Ln s = \ln |s| + i \Arg(s) 
	, \quad
	-\pi < \Arg(s) \le \pi
\end{equation}

The main consequence of the different behaviour according to the integer or non integer nature of $\beta$ is in the possibility of expressing the function $\mu(t,\beta,\alpha)$ in some alternative ways by using the formula for the inversion of the Laplace transform.

When $\beta\not\in\Nset$ it is indeed possible to formulate $\mu(t,\beta,\alpha)$ only as 
\begin{equation}\label{eq:LTInversion1}
	\mu(t,\beta,\alpha) 
	= \frac{1}{2 \pi i} \int_{{\mathcal C}_1} e^{st} {\mathcal M}(s,\beta,\alpha)  \, ds
\end{equation}
where ${\mathcal C}_1$ is any contour leaving at the left the whole real interval $(-\infty,1]$. 

When $\beta\in\Nset$, $s^{\star}=1$ is a pole of order $\beta+1$ and it can be removed by residue subtraction according to
\begin{equation}\label{eq:LTInversion0}
	\mu(t,\beta,\alpha) = \mathop{Res} \bigl( e^{st} {\mathcal M}(s,\beta,\alpha) , 1 \bigr) + \frac{1}{2 \pi i} \int_{{\mathcal C}_{0}} e^{st} {\mathcal M}(s,\beta,\alpha)  \, ds ,
\end{equation}
where the contour ${\mathcal C}_{0}$ crosses the real axis in a point on $(0,1)$ and hence it encompasses at the left only the interval $(-\infty,0]$, with $s^{\star}=1$ remaining at the right of ${\mathcal C}_{0}$. The residue in (\ref{eq:LTInversion0}) can be evaluated according to the following result.

\begin{theorem}\label{thm:PolynomialResidue}
Let $k \in \Nset$. There exists a polynomial $P_{k}(t)$ of degree  $k$, whose coefficients depend on $\alpha$, such that
\begin{equation}
	\mathop{Res} \bigl( e^{st} {\mathcal M}(s,k,\alpha) , 1 \bigr)  = \frac{e^{t}}{k!} P_{k}(t) .
\end{equation}
\end{theorem}

\begin{proof}
After preliminarily observing that
\[
	s^{-\alpha-1} 
	= \bigl( (s-1) + 1 \bigr)^{-\alpha-1}
	= \sum_{j=0}^{\infty} \binom{-\alpha-1}{j} (s-1)^{j} 
\]
and
\[
	\bigl( \ln s \bigr)^{-k-1} 
	= (s-1)^{-k-1} \left( 1 + \sum_{n=1}^{\infty} \frac{(-1)^n}{n+1}(s-1)^{n} \right)^{-k-1}
	= (s-1)^{-k-1} \sum_{n=1}^{\infty} v_{n}^{(k)} (s-1)^{n} ,
\]
where the coefficients $v_{n}^{(k)}$ can be evaluated, thanks to the Miller's formula \cite[Theorem 1.6c]{Henrici1974}, in a recursive way as 
\[
	v_{0}^{(k)} = 1 , \quad
	v_{n}^{(k)} = \sum_{j=1}^{n} \left( \frac{k j}{n}+1\right) \frac{(-1)^{j+1} v_{n-j}^{(k)}}{j+1} ,
\]
it is immediate to write 
\[
	{\mathcal M}(s,k,\alpha) (s-1)^{k+1} = \sum_{n=0}^{\infty} c_{n} (s-1)^{n}
	, \quad
	c_{n} = \sum_{j=0}^{n} \binom{-\alpha-1}{j} v_{n-j}^{(k)} .
\]

By the Leibniz's rule for differentiation we have 
\begin{eqnarray*}
	\lefteqn{
	\frac{d^{k}}{ds^{k}} e^{st} {\mathcal M}(s,k,\alpha) (s-1)^{k+1}
	 =  \sum_{j=0}^{k} \binom{k}{j} \left( \frac{d^{k-j}}{ds^{k-j}} e^{st} \right) 
			\left( \frac{d^{j}}{ds^{j}}\sum_{n=0}^{\infty} c_{n} (s-1)^{n} \right) }
	 \hspace{1.0cm} & & \\
	&=& e^{st} \sum_{j=0}^{k} \binom{k}{j} t^{k-j} \left( j! c_{j} + \sum_{n=j+1}^{\infty} 
			\frac{n!}{(n-j)!} c_{n} (s-1)^{n-j} \right) \
\end{eqnarray*}
and hence the formula for the residues allows to compute
\begin{eqnarray*}
	\lefteqn{
	\mathop{Res} \bigl( e^{st} {\mathcal M}(s,k,\alpha) , 1 \bigr)
	 = \lim_{s\to1} \frac{1}{k!} \frac{d^{k}}{ds^{k}} e^{st} {\mathcal M}(s,k,\alpha) (s-1)^{k+1} }
	\hspace{1.0cm} & & \\
	&=& \frac{e^{t}}{k!} \sum_{j=0}^{k} \binom{k}{j} t^{k-j} j! c_{j} 
	= \frac{e^{t}}{k!} \sum_{j=0}^{k} p_{k-j} t^{k-j}, \
\end{eqnarray*}
where $p_{k-j}=c_{j}k!/(k-j)!$ and from which the proof immediately follows.
\end{proof}

In Table \ref{tbl:PolynomialResidue} we present the first few polynomials involved in the residues of the LT of $\mu(t,k,\alpha)$. Further polynomials can be easily evaluated as indicated in the proof of Theorem \ref{thm:PolynomialResidue}.


\begin{table}[ht!]
\[
	\begin{array}{c|l} \hline
		k & P_{k}(t) \\ \hline
		0 & 1 \\
		1 & t - \alpha \\
		2 & t^2 - (2\alpha-1)t + \alpha^2 \\
		3 & t^3 - 3(\alpha-1)t^2 + (3\alpha^2 - 3\alpha + 1)t - \alpha^3 \\
		4 & t^4 - 2(2\alpha-3)t^3 + (6\alpha^2 - 12\alpha + 7)t^2 - (4\alpha^3 - 6\alpha^2 + 4\alpha - 1)t + \alpha^4 \
	\end{array}
\]
\caption{First polynomials $P_{k}(t)$, $k=0,1,\dots,4$ for the evaluation or residues of the LT of $\mu(t,k,\alpha)$} 
\label{tbl:PolynomialResidue}
\end{table}

Theorem \ref{thm:PolynomialResidue} is of particular importance in view of the use of (\ref{eq:LTInversion0}) for the numerical computation of $\mu(t,\beta,\alpha)$ when $\beta \in \Nset$. An efficient method for evaluating the Volterra function is indeed based on the numerical inversion of the Laplace transform which, for stability reasons and for reducing the round-off errors, requires the use of a contour located as close as possible to the origin.

\section{The Ramanujan identity and its generalizations}\label{S:Ramanujan}

The simplest instance of the Volterra functions is for $\beta=\alpha=0$ for which the already introduced notation $\nu(t) = \mu(t,0,0)$ is usually used. In this case the Ramanujan identity \cite[Page 196]{Hardy1940} 
\begin{equation}\label{eq:RamanIdentity}
	\nu(t) = e^t - N(t) 
	, \quad
	N(t) = \int_{0}^{\infty} \frac{e^{-rt}  }{r \left[ (\ln r)^2 + \pi^2 \right] } \, dr ,
\end{equation}
allows to express $\nu(t)$ in terms of the Ramanujan function $N(t)$ whose plot is presented in Figure \ref{fig:Gen_Ram0}.

\begin{figure}[ht!]
	\centering
		\includegraphics[width=0.8\textwidth]{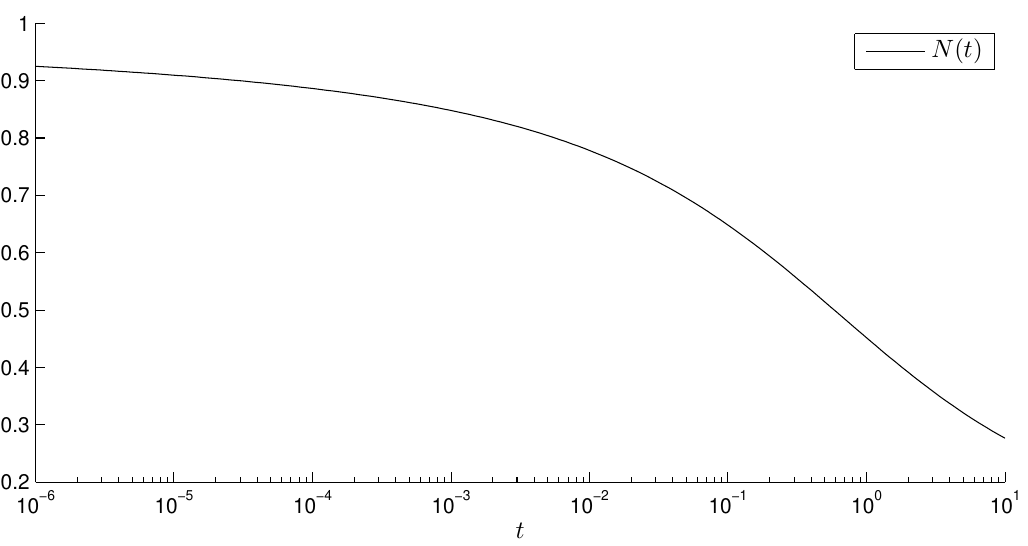} 
	\caption{Plot of the Ramanujan function $N(t)$.} 
\label{fig:Gen_Ram0}
\end{figure}


In a paper dedicated to the study of the asymptotic expansions of the solutions of some heat conduction problems \cite{Wood1966}, Wood provided a formula for the successive derivatives of $N(t)$ as
\begin{equation}
	\frac{d^n}{dt^n} N(t) = e^t - \int_{-n}^{\infty} \frac{t^u}{\Gamma(u+1)} du ,
\end{equation}
obtained on the basis of a more general relationship due to Ramanujan \cite[Eq. (11.11.1)]{Hardy1940}.  

The Ramanujan function $N(t)$ is of interest not only because it allows an alternative representation of $\nu(t)$ but it has also several important applications in physics as one can deduce from its appearance in papers on ``neutron transport theory'' \cite{DorningNicolaenkoThurber1969} and on ``slowing down of electrons'' \cite{SpencerFano1954}.

The Ramanujan identity can be extended to a slightly more general range of Volterra functions as we can prove by means of the following result.

\begin{theorem}\label{thm:RamanGeneral}
Let $k \in \Nset$, $-1<\alpha\le0$ and $t>0$. Then
\begin{equation}\label{eq:RamanGeneral}
	\mu(t,k,\alpha) 
	= \frac{e^{t}}{k!} P_{k}(t) 
	-  N_{k}(t,\alpha) ,
\end{equation}
where $P_{k}(t)$ are the polynomials introduced in Theorem \ref{thm:PolynomialResidue} and
\begin{equation}\label{eq:RamanGeneralFunction}
	N_{k}(t,\alpha) =
	\frac{1}{\pi} \int_{0}^{\infty} \frac{e^{-rt} \sin \left[ \alpha\pi + (k+1) \Arg(\ln r + i\pi) \right] }
					 {r^{\alpha+1} \left[ (\ln r)^2 + \pi^2 \right]^{\frac{k+1}{2}} }	\, dr .
\end{equation}
\end{theorem}

\begin{proof}
Consider the integral representation (\ref{eq:LTInversion0}) and, for $0<\varepsilon<1$, deform the contour ${\mathcal C}_{0}$ to an Hankel path $Q(\varepsilon)$ starting from $-\infty$ along the lower negative real axis, encircling the circle $|s| = \varepsilon < 1$ in the counterclockwise sense and returning to $-\infty$ along the upper negative real axis. By splitting $Q(\varepsilon)$ into three sub paths
 $Q(\varepsilon) = Q_{-}(\varepsilon) \cup Q_o(\varepsilon) \cup Q_{+}(\varepsilon)$ (see Figure \ref{fig:Fig_HankelPath}) with
$Q_{\pm}(\varepsilon): s=re^{\pm i\pi}$, $\varepsilon \le r < \infty$ and $Q_{o}(\varepsilon): s=\varepsilon e^{i \theta}$, $-\pi < \theta < \pi$, we can therefore write
\[
	\int_{{\mathcal C}_{0}} e^{st} {\mathcal M}(s,k,\alpha)  \, ds = \int_{Q(\varepsilon)} e^{st} {\mathcal M}(s,k,\alpha)  \, ds	= M_{-}(\varepsilon) + M_{o}(\varepsilon) + M_{+}(\varepsilon) ,
\]
where
\[
	M_{\pm}(\varepsilon) 	
	= \int_{Q_{\pm}(\varepsilon)} e^{st} {\mathcal M}(s,k,\alpha)  \, ds
	, \quad
		M_{o}(\varepsilon) 	
	= \int_{Q_{o}(\varepsilon)} e^{st} {\mathcal M}(s,k,\alpha)  \, ds.
\]

\begin{figure}[htb]
	\centering
	\includegraphics[width=0.45\textwidth]{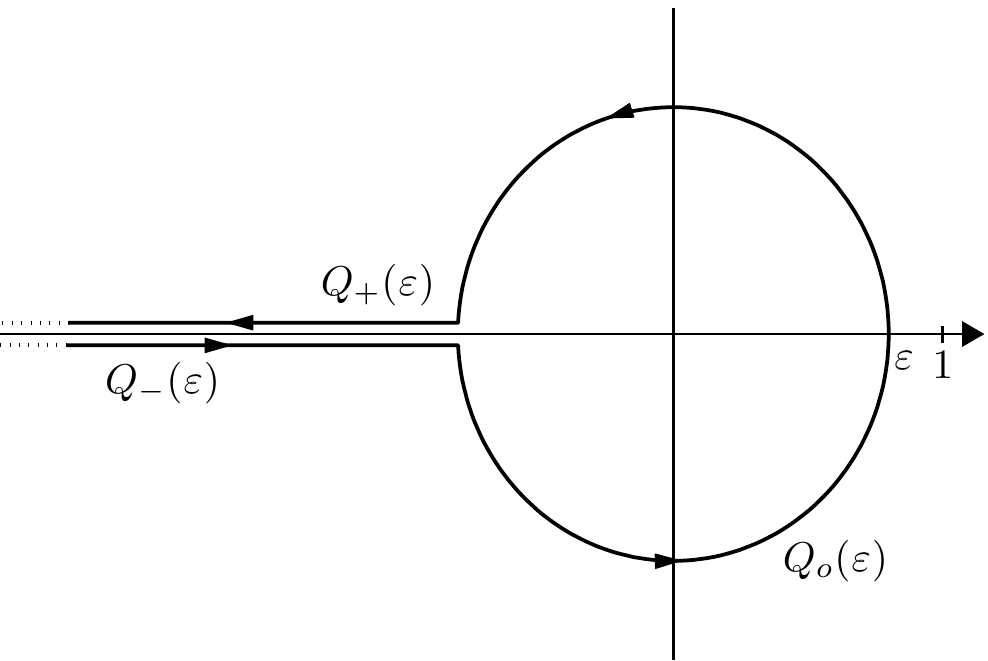} 
	\caption{Hankel path $Q(\varepsilon)$ and its three sub paths.}\label{fig:Fig_HankelPath}
\end{figure}

Since for any $z\in\Cset$ it is $z^{\alpha} = |z|^{\alpha} e^{i\alpha \Arg(z)}$, by means of (\ref{eq:ComplexLog}) we can see that
\[
	M_{\pm}(\varepsilon)
	= e^{\pm i\pi} \int_{\varepsilon}^{\infty} e^{-rt} {\mathcal M}(re^{\pm i\pi},k,\alpha) \, dr 
	= \mp \int_{\varepsilon}^{\infty} \frac{e^{-rt} e^{\mp i(\alpha+1)\pi} }{r^{\alpha+1} \left( \ln r \pm i\pi \right)^{k+1}} \, dr 
\]
and
\[
	M_{o}(\varepsilon)
	=  \varepsilon i  \int_{-\pi}^{\pi} e^{\varepsilon t e^{i\theta}} {\mathcal M}(\varepsilon e^{i\theta},k,\alpha)e^{i \theta} \, d\theta 
	= \frac{i}{\varepsilon^{\alpha}} \int_{-\pi}^{\pi} 
	 	 		\frac{e^{\varepsilon t e^{i\theta}} e^{-i \alpha \theta} }{ \left( \ln\varepsilon+i \theta\right)^{k+1}}  \, d\theta 
\]
and, under the assumption $\alpha \le 0$, it is clearly $\lim_{\varepsilon \to 0} M_{2}(\varepsilon)=0$. Therefore, as $\varepsilon \to 0$ we have
\begin{eqnarray*}
	\lefteqn{
		\lim_{\varepsilon \to 0}
		\int_{Q(\varepsilon)} e^{st} {\mathcal M}(s,k,\alpha)  \, ds
		= \lim_{\varepsilon \to 0} M_{-}(\varepsilon) + \lim_{\varepsilon \to 0} M_{+}(\varepsilon) = } \hspace{1.0cm} && \\
	&=& - \int_{0}^{\infty} 
			\frac{e^{-rt} \left[e^{i \alpha\pi} \left( \ln r + i\pi \right)^{k+1} -  e^{-i\alpha\pi} \left( \ln r - i\pi \right)^{k+1} \right] }
					 {r^{\alpha+1} \left[ (\ln r)^2 + \pi^2 \right]^{k+1}  }
			\, dr \
\end{eqnarray*}

Since it is
\[
	\left( \ln r \pm i\pi \right)^{k+1} = |\ln r \pm i\pi|^{k+1} e^{i(k+1) \Arg(\ln r + i\pi)} =
	\left((\ln r)^2 + \pi^2\right)^{\frac{k+1}{2}} e^{i(k+1) \Arg(\ln r \pm i\pi)}
\]
and, moreover, $\Arg(\ln r - i\pi) = - \Arg(\ln r + i\pi)$, we have
\begin{eqnarray*}
	\lefteqn{
		\left[ e^{i \alpha\pi} \left( \ln r + i\pi \right)^{k+1} -  
					 e^{-i\alpha\pi} \left( \ln r - i\pi \right)^{k+1} \right] = } \hspace{1.0cm} && \\
	&=& \left((\ln r)^2 + \pi^2\right)^{\frac{k+1}{2}} \left( 
				e^{i \alpha\pi + i(k+1) \Arg(\ln r + i\pi) } - e^{-i\alpha\pi - i(k+1) \Arg(\ln r + i\pi)} \right) \\
	&=& 2 i \left((\ln r)^2 + \pi^2\right)^{\frac{k+1}{2}} \sin \left[ \alpha\pi + (k+1) \Arg(\ln r + i\pi) \right] \\		
\end{eqnarray*}
and hence 
\begin{eqnarray*}
	\frac{1}{2\pi i} \int_{{\mathcal C}_{0}} e^{st} {\mathcal M}(s,k,\alpha)  \, ds 
	&=& - \frac{1}{\pi} \int_{0}^{\infty} \frac{e^{-rt} \sin \left[ \alpha\pi + (k+1) \Arg(\ln r + i\pi) \right] }
					 {r^{\alpha+1} \left[ (\ln r)^2 + \pi^2 \right]^{\frac{k+1}{2}} }	\, dr \
\end{eqnarray*}
from which we can easily conclude the proof after applying Theorem \ref{thm:PolynomialResidue}. 
\end{proof}

In order to verify that the Ramanujan identity (\ref{eq:RamanIdentity}) is actually a special case of (\ref{eq:RamanGeneral}) when $k=0$ and $\alpha=0$ we first observe that $P_{0}(t) \equiv 1$ (see Table \ref{tbl:PolynomialResidue}) and, by means of the elementary identity $\sin(\arctan(x)) = x/\sqrt{1+x^2}$, it is 
\[
	\sin\bigl[\Arg(\ln r + i\pi)\bigr] = \frac{\pi}{\sqrt{(\ln r)^2 + \pi^2 }} .
\]

Thus, when $\alpha=0$ we have 
\[
	N_{0}(t,0) 
	= \frac{1}{\pi} \int_{0}^{\infty} 
			\frac{e^{-rt} \sin \left[ \Arg(\ln r + i\pi) \right] }
					 {r \left[ (\ln r)^2 + \pi^2 \right]^{\frac{1}{2}} } \, dr 
	=	\int_{0}^{\infty} 
			\frac{e^{-rt}  } {r \left[ (\ln r)^2 + \pi^2 \right] } \, dr 
	= N(t) .
\]
We observe that the Ramanujan function $N(t)= N_{0}(t,0)$ is a {\it completely monotonic (CM) function} (i.e., for any $k\ge0$ the $k$-th derivative $N^{(k)}(t)$ exists and $(-1)^{k}N^{(k)}(t) \ge 0$ for all $t > 0$).
 Indeed, by the Bernstein theorem a necessary and sufficient condition for a function to be CM is that it is the real Laplace transform of a non-negative measure, as the integral representation (\ref{eq:RamanIdentity}) clearly shows. The CM property makes therefore $N(t)$ to be suitable as a relaxation function in dielectric and viscoelastic systems (see,  e.g., \cite{Hanyga2005,Mainardi2010}).


\begin{remark}
The assumption $\alpha \le 0$, thanks to which $M_{0}(\varepsilon) \to 0$ when $\varepsilon \to 0$, also assures the convergence of the integral in $N_{k}(t,\alpha)$. Indeed when $\alpha < 0$ this is obvious. For $\alpha=0$ we observe that $\Arg(\ln r + i\pi) \to \pi$ when $r \to 0^{+}$ and, since $k \in \Nset$, we have %
\[
	\lim_{r\to0^{+}} \sin \left[ (k+1) \Arg(\ln r + i\pi) \right]
	= \sin \left[ (k+1) \pi \right] = 0 .
\]
\end{remark}

We are able to explicitly provide a simpler representation of $N_{k}(t,0)$ for other values of $k$. Indeed, for $\alpha=0$ and $k=1$ it is
\[
	\sin \left[ \alpha\pi + (k+1) \Arg(\ln r + i\pi) \right] = \frac{2\pi\ln r}{(\ln r)^2 + \pi^2 }
\]
and hence we have
\begin{equation}\label{eq:RamanGeneralBeta1}
	N_{1}(t,0) = 
	2 \int_{0}^{\infty} \frac{e^{-rt} \ln r }
					 {r \left[ (\ln r)^2 + \pi^2 \right]^{2} }	\, dr .
\end{equation}

Moreover, for $\alpha=0$ and $k=2$, since $\sin 3 \theta = - 4 \bigl( \sin \theta\bigr)^3 + 3 \sin \theta$ we have
\[
	\sin \left[ \alpha\pi + (k+1) \Arg(\ln r + i\pi) \right] 
	= \sin \left[ 3 \Arg(\ln r + i\pi) \right]
	= \frac{-\pi^3 + 3\pi (\ln r)^2}{\left((\ln r)^2 + \pi^2\right)^{3/2} }
\]
and hence
\begin{equation}\label{eq:RamanGeneralBeta2}
	N_{2}(t,0) = 
	\int_{0}^{\infty} \frac{e^{-rt} \left(3 (\ln r)^2 -\pi^2\right)}
					 {r \left[ (\ln r)^2 + \pi^2 \right]^{3} }	\, dr .
\end{equation}

For $k=3$, since $\sin(4\arctan(x)) = 4x(1-x^2)/(1+x^2)^2$ we can see that
\begin{equation}\label{eq:RamanGeneralBeta3}
	N_{3}(t,0) = 
	4 \int_{0}^{\infty} \frac{e^{-rt} \ln r \left((\ln r)^2 - \pi^2 \right)}
					 {r \left[ (\ln r)^2 + \pi^2 \right]^{4} }	\, dr .
\end{equation}

Some of the above representations of $\mu(t,k,\alpha)$ and $\mu(t,k,0)$ were already presented in \cite{Apelblat2008} for $k=1$ and $k=2$  although without a general result as Theorem \ref{thm:RamanGeneral}. The first few instances of the generalized Ramanujan function $N_{k}(t,0)$, for $k=1,2,3$, are presented in Figures \ref{fig:Gen_Ram1}, \ref{fig:Gen_Ram2} and \ref{fig:Gen_Ram3}; the absence of the CM character when $k>0$ is evident.

\begin{figure}[ht!]
	\centering
		\includegraphics[width=0.8\textwidth]{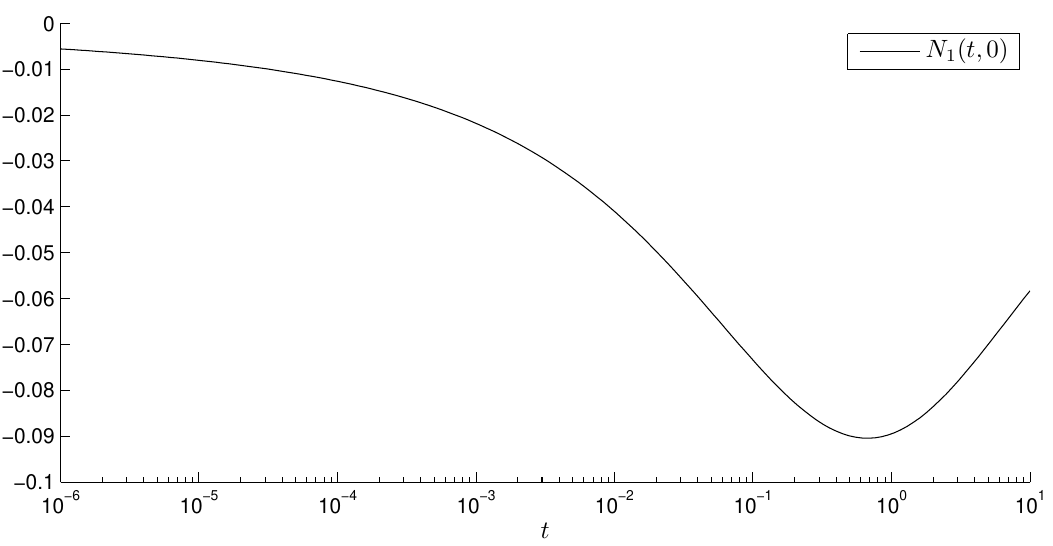} 
	\caption{Plot of the generalized Ramanujan function $N_{1}(t,0)$.} 
\label{fig:Gen_Ram1}
\end{figure}

\begin{figure}[ht!]
	\centering
		\includegraphics[width=0.8\textwidth]{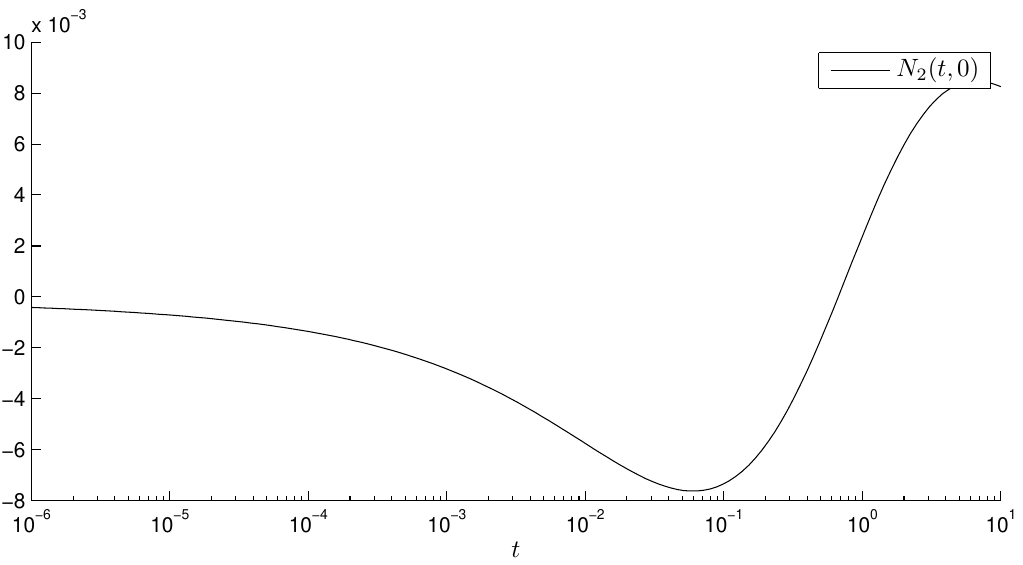} 
	\caption{Plot of the generalized Ramanujan function $N_{2}(t,0)$.} 
\label{fig:Gen_Ram2}
\end{figure}

\begin{figure}[ht!]
	\centering
		\includegraphics[width=0.8\textwidth]{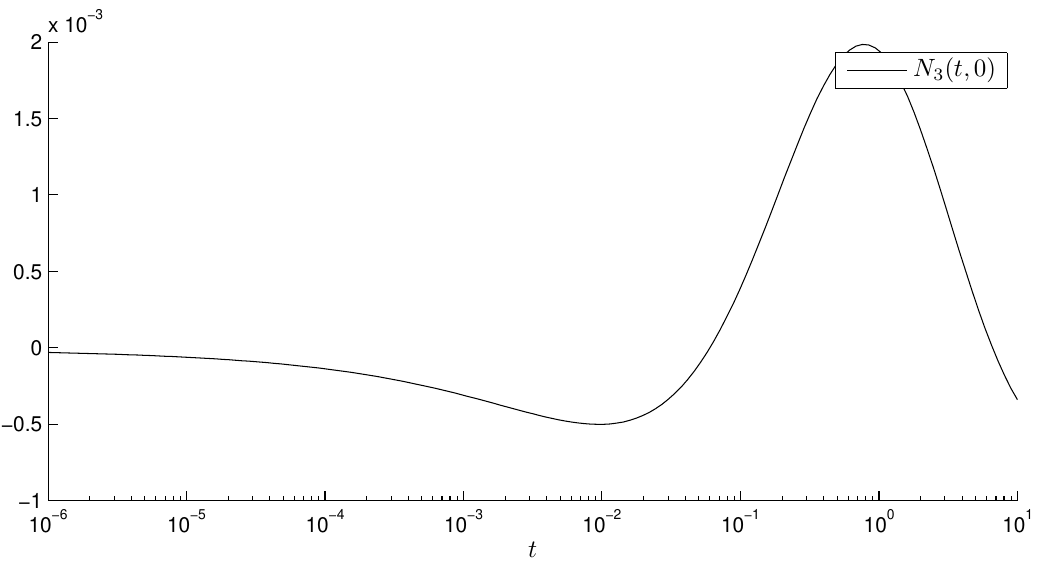} 
	\caption{Plot of the generalized Ramanujan function $N_{3}(t,0)$.} 
\label{fig:Gen_Ram3}
\end{figure}

For the numerical evaluation of the generalized Ramanujan functions we have avoided the direct computation of (\ref{eq:RamanGeneralFunction}) which involves the non trivial task of approximating integrals over an unbounded interval. We have instead preferred to invert numerically the Laplace transform by applying in (\ref{eq:LTInversion0}) the trapezoidal rule on a contour ${\mathcal C}_{0}$ of parabolic shape, according to a suitable modification of the method described in \cite{Garrappa2015_SIAM}. This approach allows to invert the Laplace transform with high accuracy and a limited computational effort on the basis of an error analysis mainly based on the distance between the contour and the singularities of the Laplace transform. The same method will be used, later on in this work, to evaluate also the Volterra function $\mu(t,\beta,\alpha)$ when $\beta \in \Nset$.

\section{Asymptotic behaviour}\label{S:Asymptotic}

Several authors in the past have faced the problem of determining asymptotic expansions for the various Volterra functions. Since the difficulties in evaluating these special functions, the knowledge of the asymptotic behaviour can provide effective information which turn out useful both for theoretical and practical purposes.

Results concerning the asymptotic expansion for small and large arguments, for the different instances (\ref{eq:VF_SpecialCases}) of the Volterra function and for the Ramanujan function (\ref{eq:RamanIdentity}), are scattered in a number of papers and books \cite{Apelblat2008,Bouwkamp1971,DorningNicolaenkoThurber1969,ErdelyiMagnusOberhettingerTricomi3,LlewllynSmith2000,WymanWong1969}. In this section we aim to collect the majority of these results and try to bring them into one more general framework.

Moreover, since an explicit representation of most of the coefficients in the asymptotic expansions is not available in a closed or simple form, we also discuss the problem of their computation by means of a numerical procedure.

We note that, although our analysis is mainly devoted to Volterra functions with real arguments $t>0$, for completeness we report the original results presented in literature for the more general case of complex arguments $z$.

\subsection{Expansion of $\mu(z,\beta,\alpha)$ for small $z$}\label{SS:mu_asy_small}

Under the assumptions that $\Re(\beta) > -1$, it was first showed in \cite{ErdelyiMagnusOberhettingerTricomi3} that $\mu(z,\beta,\alpha)$ possesses the following expansion 
\begin{equation}\label{eq:AsymptExpansionSmall}
	\mu(z,\beta,\alpha) = z^{\alpha} \sum_{n=0}^{\infty} D^{(\alpha,\beta)}_n \left( \log \frac{1}{z} \right)^{-\beta-1-n},
	\; 
	z\to 0, 
	\;
	z \in \Cset-[1,+\infty) ,
\end{equation}
where $D^{(\alpha,\beta)}_n$ are independent of $z$ and are given by
\begin{equation}
	D^{(\alpha,\beta)}_n = (\beta+1)_{n} D^{(\alpha)}_n, \quad
	D^{(\alpha)}_n = \frac{(-1)^{n}  }{n!} \mu(1,-n-1,\alpha) ,
\end{equation}
with  $(\beta)_{n}$ denoting the rising factorial
\begin{equation}
	(\beta)_{n} = 
	\frac{\Gamma(\beta+n)}{\Gamma(\beta)} = \beta(\beta+1)\cdots(\beta+n-1). 
\end{equation}

To evaluate $\mu(1,-n-1,\alpha)$, and hence $D^{(\alpha)}_n$, it is not possible to use (\ref{VolterraIntegralDefinition}) since the integral converges only for $\Re(\beta)>-1$. A repeated integration by parts \cite{ErdelyiMagnusOberhettingerTricomi3} allows, however, to extend the definition of $\mu(z,\beta,\alpha)$ to the whole $\beta$-plane according to
\begin{equation}
	\mu(z,\beta,\alpha) = \frac{(-1)^{m}}{\Gamma(\beta+m+1)} \int_{0}^{\infty} x^{\beta+m} \frac{d^m}{dx^{m}} \frac{z^{\alpha+x}}{\Gamma(\alpha+x+1)} \,dx,
\end{equation}
with $ \Re(\beta) > - m - 1$.

For integer values $\beta = -n-1$ (and hence by assuming $m=n+1$), it is simple to verify that
\begin{equation}\label{eq:mu_integer}
	\mu(z,-n-1,\alpha) = \left. (-1)^n \frac{d^n}{d x^n} \frac{z^{\alpha+x}}{\Gamma(\alpha+x+1)} \right|_{x=0}
\end{equation}
from which it is immediate to determine the coefficients $D^{(\alpha)}_n$ in terms of derivatives of the reciprocal of the gamma function as
\begin{equation}\label{eq:Dalpha}
	D^{(\alpha)}_n = \left. \frac{1}{n!}\frac{d^n}{d x^n} \frac{1}{\Gamma(\alpha+x+1)} \right|_{x=0} .
\end{equation}

To derive (\ref{eq:AsymptExpansionSmall}) it is necessary to observe from (\ref{eq:mu_integer}) that $(-1)^n\mu(1,-n-1,\alpha)$ are the coefficients in the Taylor series of $1/\Gamma(\alpha+x+1)$. Hence, by replacing this series representation in (\ref{VolterraIntegralDefinition}), together with $z^x = \exp(-x \log 1/z)$, and applying the Watson's lemma, the expansion (\ref{eq:AsymptExpansionSmall}) follows after performing a simple term-by-term integration.

As an illustrative example, in figure \ref{fig:AsyExp_Small1} we compare, for $\beta=1$, $\alpha=0.8$ and real $t>0$, the Volterra function $\mu(t,\beta,\alpha)$ and the first term $D^{(\alpha,\beta)}_n t^{\alpha} (\log 1/t)^{-\beta-1}$ in the corresponding asymptotic expansion (\ref{eq:AsymptExpansionSmall}); the reference value $\mu(t,\beta,\alpha)$ is evaluated, as explained in Section \ref{S:Ramanujan}, by numerically inverting the Laplace transform (\ref{eq:LTInversion0}) by means of a modification of the technique presented in \cite{Garrappa2015_SIAM}.

\begin{figure}[ht!]
	\centering
		\includegraphics[width=0.8\textwidth]{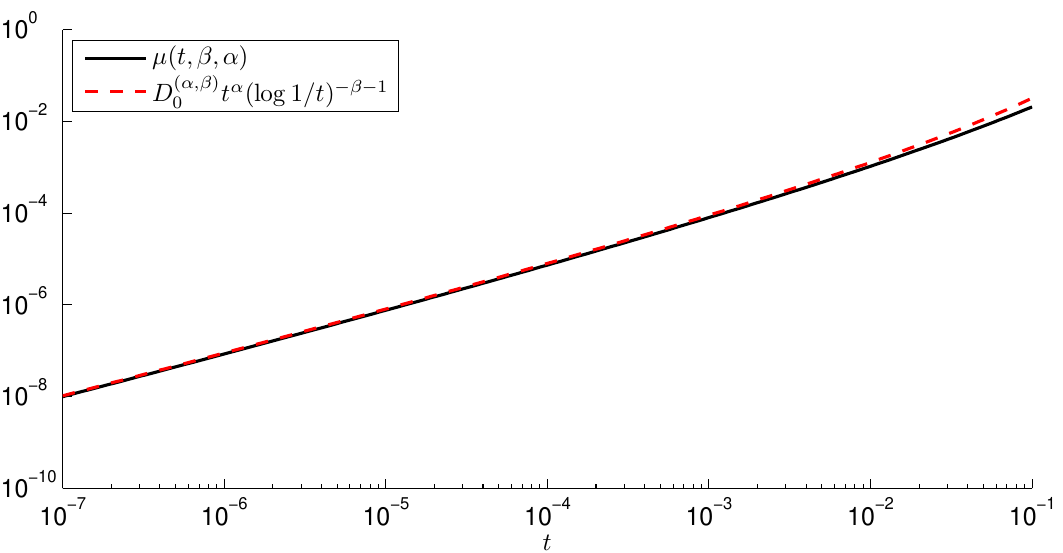} 
	\caption{Comparison between $\mu(t,\beta,\alpha)$ and the first term in the asymptotic expansion (\ref{eq:AsymptExpansionSmall}) for $\beta=1$, $\alpha=0.8$ and small $t>0$.} 
\label{fig:AsyExp_Small1}
\end{figure}

As we can clearly see, the first term in the expansion (\ref{eq:AsymptExpansionSmall}) provides a very accurate approximation of the original function as $t\to0$ and it can be used also as a tool for the numerical computation of the Volterra function; in the left plot of Figure \ref{fig:AsyExp_Small_Differ} we present, for the same parameters $\beta=1$ and $\alpha=0.8$, the difference between $\mu(t,\beta,\alpha)$ and its approximations 
\begin{equation}
	\mu_{N}(t,\beta,\alpha) = t^{\alpha} \left( \log \frac{1}{t} \right)^{-\beta-1} \sum_{n=0}^{N} D^{(\alpha,\beta)}_n \left( \log \frac{1}{t} \right)^{-n}   
\end{equation}
obtained by truncating the series (\ref{eq:AsymptExpansionSmall}) after $N$ terms. These differences, which tend to 0 in a quite fast way, indicate that few terms in the expansion (\ref{eq:AsymptExpansionSmall}) are sufficient to depict in a satisfactory way the behaviour of $\mu(t,\beta,\alpha)$ and provide an accurate approximation.

\begin{figure}[ht!]
	\centering
	\begin{tabular}{cc}
		\includegraphics[width=0.45\textwidth]{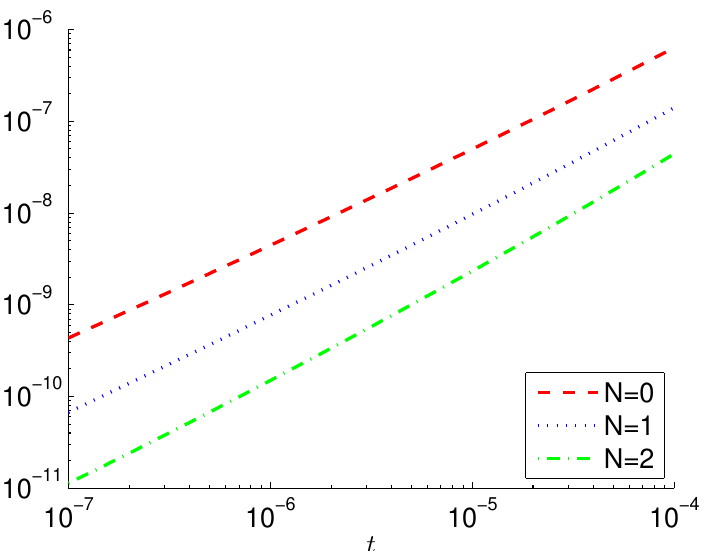} 
		&
		\includegraphics[width=0.45\textwidth]{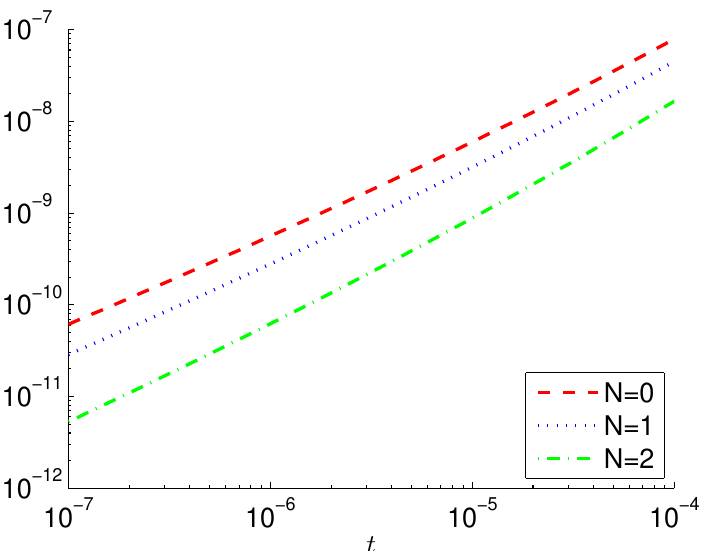} \tabularnewline
	\end{tabular}
	\caption{Difference, for small values $t>0$, between $\mu(t,\beta,\alpha)$ and some of the truncated expansions $\mu_{N}(t,\beta,\alpha)$ for $\beta=1$ and $\alpha=0.8$ (left plot) and $\beta=3$ and $\alpha=0.6$ (right plot)} 
\label{fig:AsyExp_Small_Differ}
\end{figure}

Similar results are presented in Figure \ref{fig:AsyExp_Small2}, where the values of $\mu(t,\beta,\alpha)$, for $\beta=3$ and $\alpha=0.6$, are instead investigated; the differences between $\mu(t,\beta,\alpha)$ and $\mu_N(t,\beta,\alpha)$ are presented in the right plot of Figure \ref{fig:AsyExp_Small_Differ}.

\begin{figure}[ht!]
	\centering
		\includegraphics[width=0.8\textwidth]{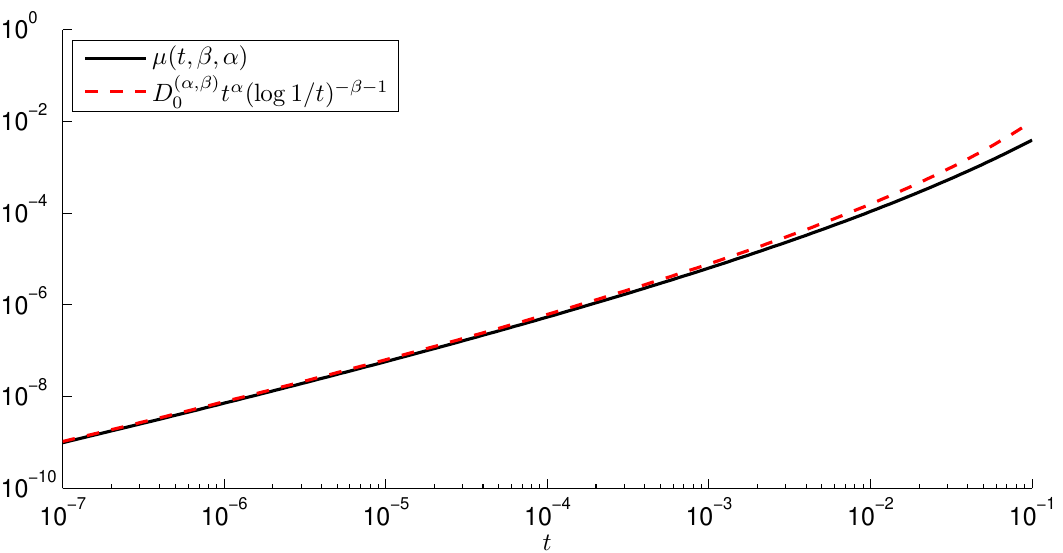} 
	\caption{Comparison between $\mu(t,\beta,\alpha)$ and the first term in the asymptotic expansion (\ref{eq:AsymptExpansionSmall}) for $\beta=3$ and $\alpha=0.6$ and small $t>0$.} 
\label{fig:AsyExp_Small2}
\end{figure}

\subsection{Expansion of $\mu(z,\beta,\alpha)$ for large $z$}

The major contribution in the study of the asymptotic behaviour of $\mu(z,\beta,\alpha)$ for large $z$ has been provided in 1969 by Wyman and Wong \cite{WymanWong1969} and was obtained by suitably deforming the Bromwich path in the formula (\ref{eq:LTInversion1}) for the inversion of the Laplace transform. In particular, it has been shown that 
\begin{equation}\label{eq:AsymptExpansionLarge}
	\mu(z,\beta,\alpha) = E(z, \beta, \alpha) + H(z, \beta, \alpha)
	, \quad
	|z| \to \infty
	, \quad 
	|\arg(z)| < \pi ,
\end{equation}
where 
\begin{equation}
	E(z, \beta, \alpha) = e^z \sum_{n=0}^{\infty} \frac{E^{(\alpha,\beta)}_{n}}{\Gamma(\beta+1-n)} z^{\beta-n}
\end{equation}
and
\begin{equation}
	H(z, \beta, \alpha) = z^{\alpha} \sum_{n=0}^{\infty} D^{(\alpha,\beta)}_n \left( \log \frac{1}{z} \right)^{-\beta-1-n} .
\end{equation}
The coefficients $E^{(\alpha,\beta)}_{n}$ in $E(z, \beta, \alpha)$ are obtained from the expansion of the generating function 
${(1-x)^{-\alpha-1} (-x)^{\beta+1}/ \bigl(\log (1-x)\bigr)^{\beta+1}}$ and can be represented as 
\begin{equation}
	E^{(\alpha,\beta)}_{n} = \frac{(-1)^n}{n!} \frac{d^{n}}{dx^{n}} \left. \frac{(1-x)^{-\alpha-1} (-x)^{\beta+1}}{\bigl( \log(1-x)\bigr)^{\beta+1}} \right|_{x=0} ,
\end{equation}
while $H(z, \beta, \alpha)$ is the same expansion introduced in (\ref{eq:AsymptExpansionSmall}) for small arguments.

Whenever $|\arg(z)|<\frac{\pi}{2}$ the dominant behaviour of $\mu(z,\beta,\alpha)$ is controlled by the exponentially large term in $E(z, \beta, \alpha)$ and the contribution of $H(z, \beta, \alpha)$ is negligible; on the contrary, when $|\arg(z)|>\frac{\pi}{2}$, the term $H(z, \beta, \alpha)$ dominates the exponentially small term $E(z, \beta, \alpha)$. In the neighborhood of $\arg(z) = \pm \frac{\pi}{2}$ a mixed type expansion is involved and the contribution of both terms $E(z, \beta, \alpha)$ and $H(z, \beta, \alpha)$ must be combined.

For $\beta=-n -1$, $n\in\Nset$, we note that $E(z, \beta, \alpha)=0$ and only a finite number of terms in $H(z, \beta, \alpha)$, namely the first $n+2$, differs from $0$, thus providing a simplified expansion for $\mu(z,-n-1,\alpha)$ according to
\begin{equation}
	\mu(z,-n-1,\alpha) 
	=  z^{\alpha} \sum_{k=0}^{n+1} (-n)_k D_{k}^{(\alpha)} 
	\left( \log \frac{1}{z} \right)^{n-k},
	\end{equation} 
for $ |z| \to \infty, \; |\arg(z)| < \pi$..

The expansion (\ref{eq:AsymptExpansionLarge}) obviously applies also to the special case $\beta=0$, for which  $\mu(z,\beta,\alpha)=\nu(z,\alpha)$, and allows to correct an erroneous expansion presented in \cite{ErdelyiMagnusOberhettingerTricomi3}. The same error was noted, in the same year 1969, in the two distinct papers \cite{DorningNicolaenkoThurber1969} and \cite{WymanWong1969}.

The comparison between $\mu(t,\beta,\alpha)$ and the first term in the asymptotic expansion (\ref{eq:AsymptExpansionLarge}) for $\beta=1$, $\alpha=0.8$ is presented in Figure \ref{fig:AsyExp_Large1}.

\begin{figure}[ht!]
	\centering
		\includegraphics[width=0.8\textwidth]{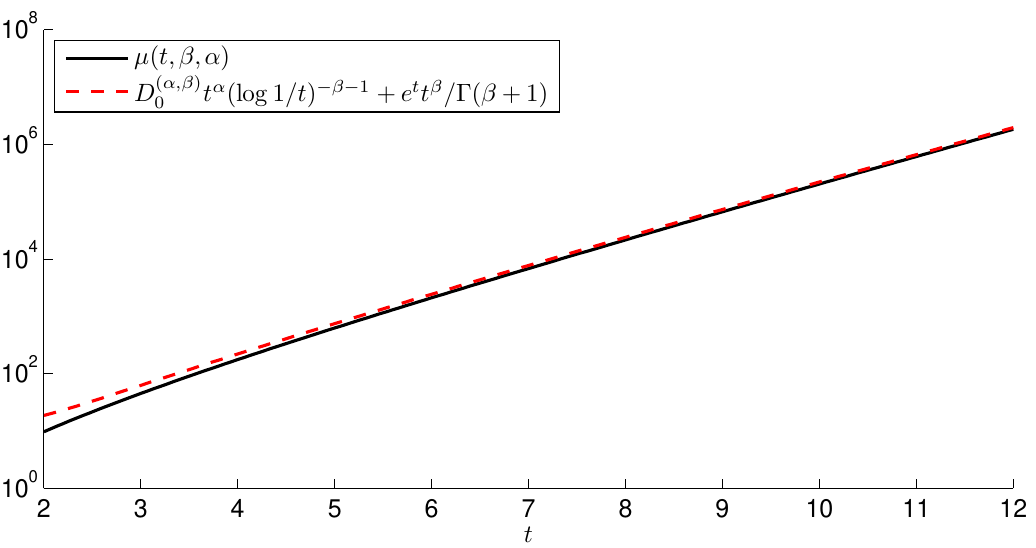} 
	\caption{Comparison between $\mu(t,\beta,\alpha)$ and the first term in the asymptotic expansion (\ref{eq:AsymptExpansionLarge}) for $\beta=1$, $\alpha=0.8$ and large $t>0$.} 
\label{fig:AsyExp_Large1}
\end{figure}

Also for large arguments it seems quite clear that as $t\to\infty$ the asymptotic expansion (\ref{eq:AsymptExpansionLarge}) provides an accurate approximation of the original function. Similar results are obtained for $\beta=3$ and $\alpha=0.6$, whose corresponding plot is presented in Figure \ref{fig:AsyExp_Large2}.

\begin{figure}[ht!]
	\centering
		\includegraphics[width=0.8\textwidth]{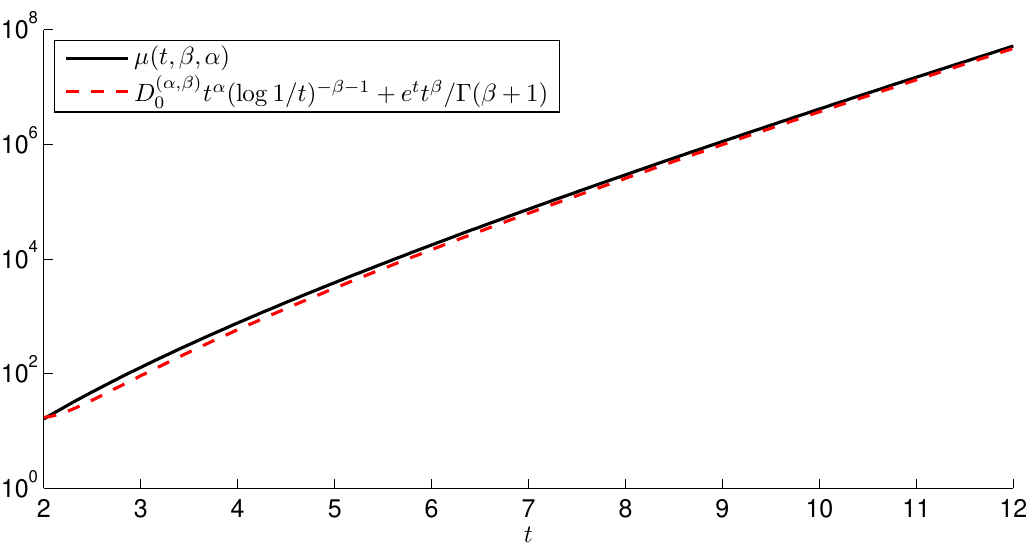} 
	\caption{Comparison between $\mu(t,\beta,\alpha)$ and the first term in the asymptotic expansion (\ref{eq:AsymptExpansionLarge}) for $\beta=3$ and $\alpha=0.6$ and large $t>0$.} 
\label{fig:AsyExp_Large2}
\end{figure}

We conclude by presenting, in Figure \ref{fig:AsyExp_Large_Differ}, the differences between $\mu(t,\beta,\alpha)$ and the approximations
\begin{equation}
\begin{array}{ll}
	\mu_{N}(t,\beta,\alpha)  
	=  & \! {\ds t^{\alpha} \sum_{n=0}^{N} D^{(\alpha,\beta)}_n \left( \log \frac{1}{t} \right)^{-\beta-1-n}}   \\
&{\ds 	+\,  e^t \sum_{n=0}^{N} \frac{E^{(\alpha,\beta)}_{n}}{\Gamma(\beta+1-n)} t^{\beta-n}},
\end{array}
\end{equation}
obtained also in this case by truncating the series (\ref{eq:AsymptExpansionLarge}) after $N$ terms. Since the large values assumed by $\mu(t,\beta,\alpha)$, in these comparisons a relative difference has been plotted. Also for large arguments, from the graphical results we can infer a fast convergence of the asymptotic representations toward the original function.

\begin{figure}[ht!]
	\centering
	\begin{tabular}{cc}
		\includegraphics[width=0.45\textwidth]{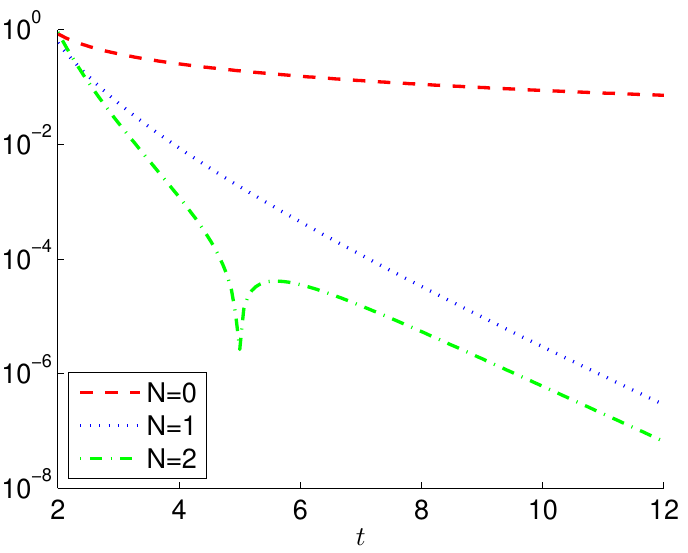} 
		&
		\includegraphics[width=0.45\textwidth]{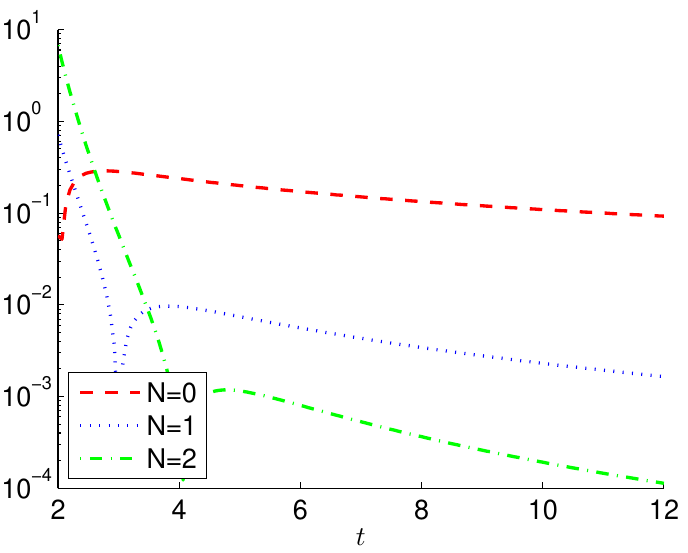} \tabularnewline
	\end{tabular}
	\caption{Relative difference, for large values $t>0$, between $\mu(t,\beta,\alpha)$ and some of the truncated expansions $\mu_{N}(t,\beta,\alpha)$ for $\beta=1$ and $\alpha=0.8$ (left plot) and $\beta=3$ and $\alpha=0.6$ (right plot)} 
\label{fig:AsyExp_Large_Differ}
\end{figure}

\subsection{Asymptotic of the Ramanujan's function}

The asymptotic behaviour of the Ramanujan's function $N(t)$ has been studied in \cite{Bouwkamp1971,DorningNicolaenkoThurber1969,LlewllynSmith2000}; this analysis is motivated not only by some important applications of this function (the work in \cite{DorningNicolaenkoThurber1969}, for instance, was stimulated by a problem in neutron transport theory) but also by the fact that, since the Ramanujan identity (\ref{eq:RamanIdentity}), the analysis of $N(t)$ allows to provide information on the special case $\nu(t)$ of the Volterra function.

For real arguments $t>0$, in 1969 Dorning, Nicolaenko and Thurber \cite{DorningNicolaenkoThurber1969} proved that 
\begin{equation}\label{eq:RamajExpansion}
	N(t) = \frac{1}{\log t} \sum_{n=0}^{\infty} \frac{D_{n}}{(\log t)^{n},}
	,\;
	t\to\infty ,
\end{equation}	
where the coefficients $D_{n}$ are obtained from the series expansion of $1/\Gamma(1-x)$
\begin{equation}
	\frac{1}{\Gamma(1-x)} = \sum_{n=0}^{\infty} D_{n}\frac{x^{n}}{n!}
\end{equation}
and hence
\begin{equation}
	D_{n} = \left.(-1)^n \frac{d^n}{dx^n} \frac{1}{\Gamma(x)}\right|_{x=1} .
\end{equation}

The expansion (\ref{eq:RamajExpansion}) was successively confirmed in 1971 by Bouwkamp \cite{Bouwkamp1971} who worked on a generalization of the function $N(t)$ which, however, does not seem of interest in this context. 

In the more recent year 2000, Llewllyn-Smith \cite{LlewllynSmith2000} have proved that (\ref{eq:RamajExpansion}) still holds for large complex arguments with $|\arg(z)|\le \frac{\pi}{2}$ (for which the integral defining $N(z)$ converges) and that, in the same right-half of the complex plane, it is also possible the alternative representation 
\begin{equation}
	N(z) = \frac{1}{\log |z|} \sum_{n=0}^{\infty} \frac{a_{n}(\theta)}{(\log |z|)^{n}},
	\;
	|z| \to \infty,
	\;
	\arg(z) \in \left[-\frac{\pi}{2},\frac{\pi}{2} \right],
\end{equation}
where $\theta = \arg(z)$ the coefficients $a_{n}(\theta)$ are now obtained from the series expansion of the modified generating function 
\begin{equation}
	\frac{e^{-i\theta x}}{\Gamma(1-x)} = \sum_{n=0}^{\infty} a_{n}(\theta)\frac{x^{n}}{n!} .
\end{equation}

From (\ref{eq:Dalpha}) it is immediate to verify that $D_n$ are the special case, for $\alpha=0$, of the coefficients $D^{(\alpha)}_n$ introduced in the Subsection \ref{SS:mu_asy_small} (i.e., $D_n=D^{(0)}_n$) and used for defining the function $H(z,\beta,\alpha)$ in the expansion (\ref{eq:AsymptExpansionLarge}). Since it is readily verified that $E(z, 0, 0)=e^{z}$,  the expansion (\ref{eq:RamajExpansion}) turns out to be a special case of (\ref{eq:AsymptExpansionSmall}). 

More generally, we observe that for $\beta = k \in \Nset$ only the first $k$ terms of the function $E$ in (\ref{eq:AsymptExpansionLarge}) differ from $0$. A numerical computation (performed as described in the following Subsection \ref{SS:EvaluationCoeffAsymExp}) has allowed us to verify that $E(t,k,\alpha)=e^tP_k(t)/k!$. As a consequence, thanks to (\ref{eq:AsymptExpansionLarge}) we are able to provide an asymptotic expansion of the generalized Ramanujan functions according to
\begin{equation}\label{eq:RamanGenerExpansion}
	N_k(t,\alpha) = t^{\alpha} \sum_{n=0}^{\infty} D^{(\alpha,k)}_n \left( \log \frac{1}{t} \right)^{-\beta-1-n} 
	, \quad t \to \infty .
\end{equation} 

We have not proved the above expansion in a formal way but we have just formulated a conjecture on the bases of the experimentally tested equivalence $E(t,k,\alpha)=e^tP_k(t)/k!$, for $k\in \Nset$. Nevertheless, since our experiments have been conducted, with positive outcomes, for a very large number of parameters $\alpha$ and $k$ and, moreover, for the limit case $\alpha=0$ the expansion is confirmed by the Ramanujan identity (\ref{eq:RamanIdentity}) together with the expansion (\ref{eq:RamajExpansion}), we think that (\ref{eq:RamanGenerExpansion}) is a reasonable conjecture whose formal proof remains however an open problem.

\subsection{Evaluation of coefficients $D^{(\alpha)}_n$ and $E^{(\alpha,\beta)}_n$}\label{SS:EvaluationCoeffAsymExp}

The evaluation of the first coefficient in both sequences $\bigl\{D^{(\alpha)}_n\bigr\}_{n\in\Nset}$ and $\bigl\{E^{(\alpha,\beta)}_n\bigr\}_{n\in\Nset}$ is immediate and indeed it is $D^{(\alpha)}_0 = 1/\Gamma(\alpha+1)$ and $E^{(\alpha,\beta)}_0=1$. For the subsequent coefficients it is necessary to evaluate, at $x=0$, the successive derivatives of the corresponding generating functions. For $\bigl\{D^{(\alpha)}_n\bigr\}_{n\in\Nset}$ computing the derivatives of the reciprocal of the Gamma function is a non simple task involving other special functions, in particular the polygamma function. Moreover, also the evaluation of the successive derivatives of the generating function of $E^{(\alpha,\beta)}_n$ can be a long and tedious task. In many circumstances it is advisable to proceed by means of a numerical approach.

The problem of computing the successive derivatives of a function $F(x)$ can be effectively solved by combining the Cauchy integral representation together with numerical quadrature. Indeed
\begin{equation}
	F_n \equiv \left. \frac{d^n}{dx^n} F(x) \right|_{x=0} = \frac{n!}{2\pi i} \int_{{\mathcal C}} \xi^{-n-1} F(\xi) d \xi ,
\end{equation}
where ${\mathcal C}$ is any simple closed contour enclosing the origin $x=0$ and lying in a region of the complex plane in which $F$ is analytic (e.g., see \cite{AblowitzFokas2003}). By assuming, for simplicity, that ${\mathcal C}$ is the circle with radius $\rho>0$ centered at the origin, the change of variable $x=\rho e^{i \theta}$ allows us to write
\begin{equation}
	F_n = \frac{n!}{2\pi \rho^n} \int_{0}^{2\pi} e^{-i \theta n} F\bigl(\rho e^{i \theta}\bigr) d \theta .
\end{equation}

Given on $[0,2\pi]$ a grid of $K$ equispaced nodes $\theta_{k} = 2 \pi k/ K$, $k=0,\dots,K-1$, the above integral can be approximated by means of a simple quadrature formula such as, for instance, the trapezoidal rule which provides the approximations
\begin{equation}
	\tilde{F}_n = \frac{n!}{K \rho^n} \sum_{k=0}^{K-1} e^{-i \theta_k n} F\bigl(\rho e^{i \theta_k}\bigr) 
\end{equation}
and efficient methods, based on the FFT algorithm, can be used to perform this computation in a very fast way. To approximate the integral in an accurate way, a quite large number of nodes can be necessary (we refer to \cite{TrefethenWeideman2014} for an insightful error analysis of the trapezoidal rule and for references to related works); we observe, however, that the same values of $F$ can be reused to approximate the whole series of coefficients $F_n$ for any $n=0,\dots,N$, thus reducing substantially the computational cost.

The radius $\rho$ of the circle on which performing the integration must be selected in order to guarantee the analyticity of $F(x)$, assure fast convergence (with the aim of keeping computational cost at minimum) and control round--off errors. An in-depth discussion on how to select the radius $\rho$ with the aim of reducing round-off errors can be found in \cite{Bornemann2011}; usually a small radius $\rho$ involves fast convergence but higher round-off errors; thus it is necessary to balance convergence and errors. The generating function ${1}/{\Gamma(\alpha + x +1)}$ of $D^{(\alpha)}_n$ is analytic on the whole complex plane while for the generating function of $E^{(\alpha,\beta)}_n$ it is necessary to impose $\rho < 1$. 

Experimentally, a radius $\rho\approx1/2$ has turned out to be satisfactory enough in both cases. Moreover, as the order $n$ of the derivative increases, it is necessary to increase also $\rho$. For this reason it is advisable to select the radius $\rho$ and the number of nodes $K$ on the basis of the higher order of the derivative which one is interested to compute.

\section{Concluding remarks}\label{S:ConcludingRemarks}

In this paper we have presented a general review of the Volterra functions. In particular, after providing an historical overview, our investigation has focused with the Laplace transform; in particular we have provided an explicit representation of the residue of the Laplace transform, thus allowing a larger degree of freedom in the choice of the contour for the numerical evaluation of the Volterra function by inversion of its Laplace transform.

We have also examined the relationship with integrals of Ramanujan type, which present important and interesting applications in physics, and provided a generalization of these integrals to any integer value of the parameter $\beta$, a result originally presented in \cite{Apelblat2008} only for few instances of $\beta$.

Finally, our attention has been devoted to present, under a quite general framework, the analysis on the asymptotic behaviour for small and large arguments: both topics are of interest for the practical computation of these special functions; moreover, we have illustrated the main results on the asymptotic behaviour by means of some plots obtained on the basis of the numerical inversion of the Laplace transform.

As an instructive example, in the Appendix we have shown the way in which the original problem formulated by the Italian mathematician Vito Volterra can be solved by means of the Volterra functions. 

We think that the possible applications of Volterra functions, in particular for describing and studying non local operators with slow varying kernels of logarithmic type is a not completely explored field.

The aim of this paper is to provide to the academic community a comprehensive knowledge  and a list of references for possible applications of Volterra functions.

\section*{Appendix}
In this Appendix we  use the technique of Laplace transform to solve the Volterra equation of the first kind 
$$ 	\int_0^t u(\tau) \, \log (t-\tau)\, d\tau = f(t)\,,
	\leqno(A.1)$$
	where $f(t)$  is an assigned, sufficiently well-behaved  function with $f(0)=0$.
	It is an instructive exercise in order to derive in the simplest way the solution obtained originally by Volterra \cite{Volterra1916} which has lead 
	to the  function $\nu(t)$ 	using  a cumbersome notation.
	
We adopt the sign $\div$ to denote the juxtaposition of the function $f(t)$
with its Laplace transform according to 
$$f(t) \,\div\, \widetilde f(s):=  {\mathcal{L}}\, \left[ f(t);s \right]  := \int_0^\infty \!\!
   e^{-st}\, f(t)\, dt\,.$$
To solve Eq. (A.1) we need to recall the following Laplace transform pairs
$$ \log(t) \, \div\,  -\frac{\log \, s + \gamma}{s}\,, \leqno(A.2)$$
where $\gamma  = 0.57721...$ is the Euler-Mascheroni constant, and 
$$\nu(t)  
:= \int_0^\infty \frac{t^u}{\Gamma(u+1)}\,du
\, \div \, \frac{1}{ s \, \log \,s}\,. \leqno(A.3)$$
As a consequence of the convolution the solution of Eq (A.1) in the Laplace domain reads
 $$\widetilde u(s) =  -\frac{s \widetilde f(s)}{\log s +\gamma} \,.\leqno(A.4)$$ 
Now $s\widetilde f(s)$
can be interpreted as the Laplace transform of the first derivative of $f(t)$ whereas
  $$\log s +\gamma = \log \left(s \, e^\gamma\right )
  \div e^{-\gamma} \, \dot \nu \left(t\, e^{-\gamma}\right) \,, \leqno (A.5)$$
  being $\nu(0)=0$.
  Above we have used the scaling property for a generic function $g(t)$ applied to 
  $\dot \nu(t)$
   $$ \frac{1}{c}  g\left (\frac{t}{c}\right) \,\div\, \widetilde g(c s)\,, \quad c>0\,.$$ 
  Then 
   $$ u(t) = - e^{-\gamma} \, \int_0^t \!\!
  \dot f (t-\tau) \, \dot \nu\left(\tau e^{-\gamma}\right)\, d\tau\,, \leqno(A.6)   
  $$
  in agreement with that anticipated in the Introduction.
  
  We can recast the   Volterra function $\nu(t)$  if we require $f(t)$ be twice differentiable  with $\dot f(0) = f(0) =0$.
  In fact, rewriting    Eq. (A.4) as
  $$\widetilde u(s) =  -\frac{s^2 \,\widetilde f(s)}{s \,(\log s +\gamma)} 
   = -\frac{s^2 \,\widetilde f(s)}{s\,\log \left( s \,e^\gamma \right)}
  \,,\leqno(A.7)$$
  we get 
  $$ u(t) = - e^{-\gamma} \, \int_0^t \!\!
  \stackrel{\cdot \cdot} {f} (t-\tau) \,  \nu\left(\tau e^{-\gamma}\right)\, d\tau\,.
  \leqno(A.8)   
  $$
  
  In the particular case $f(t)= t$, we obtain from Eq. (A.6) and from the definition of
   $\nu(t)$ in Eq. (A.3)
   $$ u(t) = -\int_0^\infty \frac{t^u\, e^{-\gamma u}}{\Gamma(u+1)}\,du\,.
   \leqno(A.9)$$

\section*{Acknowledgements}
The authors appreciate constructive remarks and suggestions of the referees  that helped to improve the manuscript.

Furthermore, F. Mainardi  likes to thank Prof. Alexander Apelblat for providing him  with the copies of his books and for keeping a correspondence via e-mail on the matters related to Volterra functions. Without this correspondence, this survey paper would not have conceived. Let us note that Apelblat devoted his attention to Volterra functions since several years as a Chemist Engineer, not as a Mathematician. His current position is Emeritus Professor at the  Chemical Engineering Department, Ben-Gurion University of the Negev, Beer Sheva, Israel.

The work of R. Garrappa is supported by the INdAM-GNCS under the project 2015 ``Metodi numerici per problemi di diffusione anomala''. The work of F. Mainardi has been carried out in the framework of the activities 
of the National Group of Mathematical Physics (INdAM-GNFM) and of the Interdepartmental Center ``L. Galvani'' for integrated studies of Bioinformatics, Biophysics and Biocomplexity of the University of Bologna.


\bibliographystyle{plain}

\end{document}